\documentclass[12pt]{amsart}
\usepackage[top=30truemm,bottom=30truemm,left=25truemm,right=25truemm]{geometry}
\usepackage{txfonts}
\usepackage{mathrsfs}
\usepackage{amsmath, amsthm, amssymb}
\usepackage{color}
\usepackage{bm}
\usepackage{amsfonts,amssymb}
\usepackage{dsfont}
\usepackage{amscd}
\usepackage{extarrows}
\usepackage{amsmath}
\usepackage{mathrsfs}
\usepackage{enumerate}
\usepackage{amscd}
\usepackage[all]{xy}
\usepackage[pagebackref,colorlinks]{hyperref}
\usepackage{geometry}
\usepackage[colorlinks]{hyperref}
\usepackage[capitalize]{cleveref}
\usepackage{tikz}              
\usetikzlibrary{arrows.meta}

\newtheorem{theorem}{Theorem}[section]

\newtheorem{lemma}{Lemma}[section]
\newtheorem{corollary}{Corollary}[section]
\newtheorem{proposition}{Proposition}[section]
\newtheorem{remark}{Remark}[section]

\newtheorem{conj}{Conjecture}[section]

\newcommand{\R}{\mathbb{R}} 
\newcommand{\C}{\mathbb{C}} 
\newcommand{\Z}{\mathbb{Z}} 

\newcommand{\be}{\begin{equation}}
	\newcommand{\ee}{\end{equation}}
\newcommand{\bea}{\begin{eqnarray}}
	\newcommand{\eea}{\end{eqnarray}}
\newcommand{\ben}{\begin{eqnarray*}}
	\newcommand{\een}{\end{eqnarray*}}
\newcommand{\bt}{\begin{split}}
	\newcommand{\et}{\end{split}}
\newcommand{\bet}{\begin{equation}}

\begin{document}
		\title[SOS conjecture]{Macaulay representation of the prolongation matrix and the SOS conjecture}

		\author[Z. Wang]{Zhiwei Wang}
		\address{Zhiwei Wang: Laboratory of Mathematics and Complex Systems (Ministry of Education)\\ School of Mathematical Sciences\\ Beijing Normal University\\ Beijing 100875\\ P. R. China}
		\email{zhiwei@bnu.edu.cn}
		
		\author[C. Yue]{Chenlong Yue}
		\address{Chenlong Yue: School of Mathematical Sciences\\ Beijing Normal University\\ Beijing 100875\\ P. R. China}
		\email{aystl271828@163.com}
		
		\author[X. Zhou]{Xiangyu Zhou}
		\address{Xiangyu Zhou: Institute of Mathematics\\Academy of Mathematics and Systems Science\\and Hua Loo-Keng Key
			Laboratory of Mathematics\\Chinese Academy of
			Sciences\\Beijing\\100190\\P. R. China}
		\address{School of
			Mathematical Sciences, University of Chinese Academy of Sciences,
			Beijing 100049, P. R. China}
		\email{xyzhou@math.ac.cn}
		
		\begin{abstract}
				Let $z \in \mathbb{C}^n$, and let $A(z,\bar{z})$ be a  real valued diagonal bihomogeneous Hermitian polynomial such that $A(z,\bar{z})\|z\|^2$ is a sum of squares, where $\|z\|$ denotes the Euclidean norm of $z$.  In this paper, we provide an estimate for the rank of the sum of squares $A(z,\bar{z})\|z\|^2$ when $A(z,\bar{z})$ is not semipositive definite. As a consequence, we confirm the SOS conjecture proposed by Ebenfelt for $4 \leq n \leq 6$ when $A(z,\bar{z})$ is a real valued diagonal (not necessarily bihomogeneous) Hermitian polynomial, and we also give partial answers to the SOS conjecture for $n\geq 7$.
			
		\end{abstract}
		

		\maketitle
		
		\tableofcontents

		\section{Introduction}
		Let $z=(z_1,\cdots, z_n)$ be the complex coordinates of $\mathbb C^n$, and let  $||z||$ be  the usual Euclidean norm. Let $A(z,\bar z)\in \C[z_1,\cdots, z_n,\bar z_1,\cdots, \bar z_n]$ be a real valued Hermitian polynomial. Let 
		$$\mathfrak{Z}=(1,z_1,\cdots,z_n,z_1^2,\cdots,z_1z_n,\cdots,z_n^d)$$
		be a basis of the polynomials in $z$ of degree at most $d$ in left lexicographic order.
		Then there exists a Hermitian matrix $H$ such that $A(z,\bar{z})=\mathfrak{Z}H\mathfrak{Z}^*$, where $\mathfrak{Z}^*$ is the conjugate transpose of $\mathfrak{Z}$. 
		When we refer to properties such as eigenvalues and rank of a Hermitian polynomial, we actually mean the corresponding matrix; conversely, when talking about the properties of the matrix, we can also associate them with the Hermitian polynomial.
		
		
		Originating from Hilbert's 17th problem, a vast amount of literature has addressed the question of whether non-negative polynomials can be represented as sums of squares. The study of Hermitian polynomials using matrices has also been extensively discussed in the literature, see for example \cite{D1}. An important fact is that $A(z,\bar z)$ is an sum of squares is equivalent to the corresponding matrix $H$ of $A(z,\bar z)$ being positive semi-definite. A famous theorem by Quillen \cite{Q1} states that if a homogeneous Hermitian polynomial $A(z,\bar z)$ is strictly positive outside the origin, then there exists a positive integer $N$ such that $A(z,\bar{z})||z||^{2N}$ is positive definite. We call $A(z,\bar{z})||z||^{2N}$ the $N$-th prolongation of $A$. Positive definiteness does not always occur immediately, and it is entirely possible for the matrix to be positive semi-definite during the prolongation process. We are particularly interested in polynomials that become positive semi-definite after the first prolongation (i.e. $A(z,\bar z)\|z\|^2$ can be written as a sum of norms squares of holomorphic polynomials) and the possible ranks after prolongation. 
		The following conjecture, named the Sums of Squares (SOS) conjecture, was proposed by Ebenfelt.
		
		\begin{conj}[{\cite[Conjecture 1.2]{E1}}]
			For $n \geq 2$,\ if the  real valued Hermitian polynomial $A(z,\bar z)$ becomes positive semi-definite after the first prolongation, then the rank $R$ of $A(z,\bar z)||z||^2$ either satisfies
			
			\begin{equation}
				R \geq  (\kappa_0+1)n -\frac{(\kappa_0+1)\kappa_0}{2}-1.
				\label{eq1-1}
			\end{equation}
			Here $\kappa_0$ is the largest integer such that $\kappa(\kappa+1)/2 < n$.
			
			Or there exists $\kappa \in \{ 0,1,2,\cdots,\kappa_0 \}$ such that
			\begin{equation}
				\kappa n-\frac{\kappa(\kappa-1)}{2} \leq R \leq \kappa n.
				\label{eq1-2}
			\end{equation}
			\label{conjecture1-1}
		\end{conj}
		\begin{remark}
			The SOS conjecture is motivated by Huang's lemma \cite{H1} and the Huang-Ji-Yin Gap conjecture \cite{HJY} on the rational proper maps between the complex unit balls. Actually, using a CR version of the Gauss equation, Ebenfelt \cite{E1} proved that the Gap conjecture is a consequence of the SOS conjecture.
		\end{remark}

		There are substantial evidences suggesting that this conjecture is true:\begin{itemize}
			\item When $A(z,\bar z)$ itself is a sum of squares, (\ref{eq1-2}) was proven by Grundmeier and Halfpap \cite{G3} even before the conjecture was proposed.  
			\item When $n=2$,  a lemma by Huang \cite{H1} showed that either $R=0$ or $R \geq n$, which proves the conjecture.
			\item  When $n=3$ and $A(z,\bar z)$ is diagonal, Brooks and Grundmeier \cite{B1} proved  that the SOS conjecture holds. 
			\item  More recently, Y. Gao and S. Ng \cite{G1} made a breakthrough, using geometric methods to demonstrate the existence of gaps under more general conditions. Subsequently, these methods were further developed to make progress to the Huang-Ji-Yin Gap Conjecture \cite{GY}.
			
		\end{itemize}

		Based on the Grundmeier--Halfpap result, Ebenfelt pointed out that an optimistic view of the situation in the conjecture would be to hope that the ``gaps'' in linear ranks predicted in \eqref{eq1-2} can only occur when \( A(z,\bar{z}) \) is itself an SOS. Furthermore, if \( A(z,\bar{z}) \) is not an SOS but \( A(z,\bar{z})\|z\|^2 \) is still an SOS, then the lower bound \eqref{eq1-1} always holds. This is named the \emph{weak (alternative) sum-of-squares conjecture}. If true, it implies the SOS conjecture in view of the Grundmeier-Halfpap result.
		\begin{conj}[Weak (Alternative)   SOS conjecture, {\cite[Conjecture 1.5]{E1}}]\label{conj: weak sos}
			If $A(z,\bar z)$ is not a sum of squares but $A(z,\bar z)\|z\|^2$ is a sum of squares, then \eqref{eq1-1} holds.   
		\end{conj}
		
		As pointed out by Ebenfelt \cite{E1}, one of the main difficulties in  \cref{conj: weak sos} comes from the fact that it seems hard to characterize when $A(z,\bar z)\|z\|^2$ is in fact an SOS.
		
		In the present paper, we study the \cref{conj: weak sos} for real valued diagonal Hermitian polynomials.
		
		We call a real valued $(d,d)$-bihomogeneous Hermitian polynomial $A(z,\bar{z})$ in $\mathbb{C}^n$ a $d$-form, denoted as $A(z,\bar{z})\in P_{n,d}$. If such a polynomial $A(z,\bar{z})$ can be expressed as the sum of squares of several holomorphic polynomials, i.e., there exists a holomorphic polynomial mapping $h(z)=(h_1(z),\cdots,h_R(z))$ such that $A(z,\bar{z})=||h||^2$, we denote $A \in \mbox{SOS}_n$,\ let $\Sigma_{n,d}:= \mbox{SOS}_n \cap P_{n,d}.$  
		
		By introducing the prolongation map and its Macaulay-type representation, we can characterize \( A(z,\bar{z})\|z\|^2 \) as a sum of squares (\cref{prop: nonneg crit}). Then, by employing an inductive argument, Macaulay-type estimates, and meticulous counting, we first obtain the following estimate for the case of diagonal bihomogeneous Hermitian polynomials.
		
		\begin{theorem}\label{thm: main 1}
			If the matrix corresponding to the real valued  $(d,d)$-bihomogeneous Hermitian polynomial $A(z,\bar{z})$ is diagonal, and $A(z,\bar{z}) \notin \Sigma_{n,d}$ while $A(z,\bar{z})||z||^2 \in \Sigma_{n,d+1}$, then
			
			(A) For all $n,d \geq 2$,\ the rank of $A(z,\bar{z})||z||^2$ satisfies
			\begin{align*}
				R \geq 3n-4.
			\end{align*}
			When $2 \leq n \leq 6$, $3n-4 \geq  (\kappa_0+1)n -\kappa_0(\kappa_0+1)/2 -1$, where $\kappa_0$ is the largest integer such that $\kappa(\kappa+1)/2 < n$. This lower bound is tight for $2 \leq n \leq 4$.
			
			(B) For all $n \geq 6$ and $d=2$,\ the rank of $A(z,\bar{z})||z||^2$ satisfies
			\begin{align*}
				R \geq \frac{n^2+n}{2}-6.
			\end{align*}
			In terms of order of magnitude, $O((n^2+n)/2-6)=1/2+ O((\kappa_0+1)n -\kappa_0(\kappa_0+1)/2 -1)$, where $\kappa_0$ is the largest integer such that $\kappa(\kappa+1)/2 < n$.  
		\end{theorem}
		
		\begin{remark}
		For $2 \leq n \leq 12$, (A) is closer to the lower bound of Ebenfelt's conjecture than the lower bound obtained by Gao and Ng \cite{G1}. Direct calculation shows that (A) matches the lower bound of Ebenfelt's conjecture for $2 \leq n \leq 6$. 
		\end{remark}

		Observing that any real valued diagonal Hermitian polynomial can be decomposed into sums of real valued diagonal bihomogeneous Hermitian polynomials, \cref{thm: main 1} can be immediately generalized to the case of diagonal Hermitian polynomials.
		
		\begin{corollary}\label{Eben's conj low dim}
			For all $n \geq 2$, if the matrix corresponding to the real valued Hermitian  polynomial $A(z,\bar{z})$ is diagonal, and $A(z,\bar{z}) \notin \mbox{SOS}_n $ while $ A(z,\bar{z})||z||^2 \in \mbox{SOS}_n,$ then the rank of  $A(z,\bar{z})||z||^2$ satisfies
			\begin{align*}
				R \geq 3n-4.
			\end{align*}
			And Ebenfelt's conjecture holds if $n \leq 6$ for such polynomials.
		\end{corollary}
	
	\begin{remark}
		As we have already mentioned, for $n = 2$ this was proved by Huang \cite{H1}, and for $n = 3$ by Brooks-Grundmeier \cite{G1}; for $4 \leq n \leq 6$, it establishes the conjecture for the first time.
		\end{remark}
		\begin{remark}

			Our method is more inclined towards the elementary algebraic approach used by Macaulay, as opposed to the commutative algebra method of Brooks-Grundmeier and the geometric method of Gao-Ng. It is interesting to get geometric applications of   our results in the study of proper holomorphic mappings between unit balls. It is also interesting to interpret our proof in combinatorial terms.
		\end{remark}

		
		\section{Preliminaries}
		In this section, we first introduce the Macaulay estimate and the Macaulay representation of the prolongation matrix , which can be used to give an characterization of the property of SOS.  Then we reduce the study of \cref{conj: weak sos} to the estimate of the rank of the first prolongation of a real homogeneous monomials.

		Now given a field $k$ of characteristic 0, we denote the graded polynomial ring $k[x_1,\cdots,x_n]$ by  
		\[ P = \bigoplus_{d \geq 0} P_d, \]
		where $P_d$ is the linear space spanned by all homogeneous polynomials of degree $d$. A subspace $A_d \subset P_d$ is called a $P_d$-monomial space if it can be linearly spanned by monomials in $ P_d$. Set $|A_d| = \dim_k(A_d)$. It is clear that 
		\begin{align*}
			|P_d| &= \binom{n+d-1}{d}, \quad \forall d \geq 0. \label{eq:dim_P_d}
		\end{align*}
		
		Let $(A_d)$ be the graded monomial ideal generated by $A_{d}$ in $P$. In order to correspond with the Macaulay representation of numbers, let  $A_{d}^{\langle 1 \rangle}:=(A_d)\cap P_{d+1}$, then $A_{d}^{\langle 1 \rangle}\subset P_{d+1}$ is a $P_{d+1}$-monomial space.

		The support set supp$(A(x))$ of a polynomial $A(x)$ is defined as the set of multi-exponents of the monomials corresponding to its non-zero coefficients. For example, if $A(x)=\sum_\alpha a_\alpha x^\alpha$, the support set supp$(A(x))=\{\alpha, a_\alpha\neq 0 \mbox{~in~} A(x)\}$. Denote by $|\mbox{supp}(A(x))|$ the number of the elements in supp$(A(x))$.
		
		In this paper, we are concerned with the  support set of the polynomial $A(x)S_1(x)$, where $A(x)=\sum_{|\alpha|=d} a_{\alpha}x^{\alpha}$ and  $S_1(x)=x_1+\cdots +x_n$.  Let $A_d$ be the $P_d$-monomial space generated by the monomials corresponding to the non-zero coefficients of the polynomial $A(x)$. 
		
		It is obvious that if the polynomial $A(x)$ has non-negative coefficients,  then 
		\[|\mathrm{supp}((A(x) S_1(x)))|=|A_d^{\langle 1 \rangle}|.\]
		Polynomials with non-positive coefficients also have the same property.

		\subsection{Macaulay's estimate}

		Fix a positive integer $N$. For any $d \in \mathbb{N}^+$, there exists a unique sequence of positive integers $k_d > k_{d-1} > \cdots > k_{\delta} \geq \delta \geq 1$ such that 
		\[ N = \binom{k_d}{d} + \binom{k_{d-1}}{d-1} + \cdots + \binom{k_{\delta}}{\delta}. \]
		This is called the $d$-Macaulay representation of $N$. We define
		\[ N^{\langle d \rangle} := \binom{k_d+1}{d+1} + \binom{k_{d-1}+1}{d} + \cdots + \binom{k_{\delta}+1}{\delta+1}, \]
		where $(N)^{\langle d \rangle} \leq (N+1)^{\langle d \rangle}$.  For convenience of presentation, we set $0^{\langle d \rangle} = 0$. More discussions with   the Macaulay representation of integers are referred to  \cite{G2}.
		
		\begin{theorem}[Macaulay's estimate \cite{M1}]\label{thm: Mac estimate}
			For any $n,d \geq 1$, $A_d$ is a monomial space in $P_d$, and $\mathrm{codim}(A_d)$ is its codimension. Then the codimension  of $A_d^{\langle 1 \rangle}$ in $P_{d+1}$ satisfies
			\begin{align*}
				\mathrm{codim}(A_d^{\langle 1 \rangle}) \leq \mathrm{codim}(A_d)^{\langle d\rangle}.
			\end{align*}
			
			This upper bound is sharp and is achieved by the left lexicographic order space.
		\end{theorem}
		
		\subsection{Prolongation map}
		
		For any $d \in \mathbb{N}$, the prolongation map  from $P_d$ is defined as
		\begin{align*}
			J_{n,d}: P_d &\rightarrow P_{d+1}, \\
			A(x) &\mapsto A(x) S_1(x).
		\end{align*}
		The prolongation map is a linear mapping. The left lexicographic basis vectors of $P_d$ are denoted as the row vector $\mathfrak{X}_d^n = (x_1^d, x_1^{d-1}x_2, \cdots, x_n^d)$. Let $\bm{h}$ and $\hat{\bm{h}}$ be the coordinate vectors of $A$ and $J_{n,d}(A)$ with respect to the left lexicographic basis $\mathfrak{X}_d^n$ and $\mathfrak{X}_{d+1}^n$, respectively. Under these bases, there exists a unique matrix, also denoted by $J_{n,d}$, such that
		\begin{align*}
			\hat{\bm{h}} = J_{n,d} \bm{h}.
		\end{align*}
		\begin{remark}
			Unless otherwise stated, all vectors in this paper are default column vectors, denoted by bold letters.
		\end{remark}
		
		The prolongation matrix has a recurrence relation with respect to its subscripts $n$ and $d$, which is the foundation of  our proof.	\begin{proposition}\label{prop: basic of pro}
			The matrix $J_{n,d}$ has $\binom{n+d}{d+1}$ rows and $\binom{n+d-1}{d}$ columns. Its elements are 0 or 1. Each row has at least 1 and at most $n$ non-zero elements, and each column has exactly $n$ non-zero elements. It is appropriately assembled from $J_{n,d-1}$, $J_{n-1,d}$, and the identity matrix $I$ of order $\binom{n+d-2}{d}$ as follows:
			\begin{align*}
				J_{n,d} = 
				\left( 
				\begin{array}{@{}c|c@{}}
					& 0 \\ 
					J_{n,d-1} & \vdots \\ 
					& 0 \\ 
					& I \\ 
					\hline
					0 & J_{n-1,d} 
				\end{array} 
				\right).
			\end{align*}

		\end{proposition}
		\begin{proof}
			By direct computation, we have
			\begin{align*}
				A(x)S_1(x)=&[x_1(\sum_{\alpha_1 \geq 1} a_{\alpha} x_1^{\alpha_1-1}\cdots x_n^{\alpha_n})+\sum_{\alpha_1=0 } a_\alpha x_2^{\alpha_2}\cdots x_n^{\alpha_n}]\cdot (x_1+\cdots + x_n)\\
				=&x_1(\sum_{\alpha_1 \geq 1} a_{\alpha} x_1^{\alpha_1-1}\cdots x_n^{\alpha_n})\cdot (x_1+\cdots + x_n) + x_1 (\sum_{\alpha_1=0 } a_\alpha x_2^{\alpha_2}\cdots x_n^{\alpha_n}) \\
				&+ (\sum_{\alpha_1=0 } a_\alpha x_2^{\alpha_2}\cdots x_n^{\alpha_n})(x_2+\cdots +x_n).
			\end{align*}
			Set $\bm{h}=(\bm{h}', \bm{h}_d)$, where $\bm{h}'$ and $\bm{h}_d$ are the coordinate vectors of two component polynomials of $A(x)$:  $\sum_{\alpha_1 \geq 1} a_{\alpha} x_1^{\alpha_1}\cdots x_n^{\alpha_n}$ and $\sum_{\alpha_1=0 } a_\alpha x_2^{\alpha_2}\cdots x_n^{\alpha_n}$, respectively. Note that the coordinate vector $\bm{h}'$ corresponds to the coordinate vector of $\sum_{\alpha_1 \geq 1} a_{\alpha} x_1^{\alpha_1-1}\cdots x_n^{\alpha_n}$ with respect to the basis $\mathfrak{X}_{d-1}^n$ of the space  $P_{d-1}$. The polynomial decomposition of $A(x)S_1(x)$ above is equivalent to the following column vector decomposition:
			$$
			\hat{\bm{h}}= \left(  \begin{array}{c}
				J_{n,d-1}\bm{h}' \\
				\hline
				0
			\end{array}\right)
			+ \left(  \begin{array}{c}
				0\\
				\bm{h}_d\\
				\hline
				0
			\end{array}\right)+ \left(  \begin{array}{c}
				0 \\
				\hline
				J_{n-1,d}\bm{h}_d
			\end{array}\right).
			$$
			Converting the above equation into the multiplication of block matrices proves the recurrence formula of $J_{n,d}$. 
			Using the recurrence relation multiple times until $d$ decreases to 0, we get:
			\begin{align}\label{equ: mac repr Jnd}
				J_{n,d}=
				\begin{pmatrix}
					1 &&&&&\\
					J_{n-1,0}& I &&&&\\
					& J_{n-1,1} & I &&&\\
					&&\cdots &&I&\\
					&&&\cdots & J_{n-1,d-1} &I \\
					&&&&&J_{n-1,d}
				\end{pmatrix}.
			\end{align}
			
			Next, we prove the numerical characteristics of the matrix $J_{n,d}$ by induction. 
			
			When $n=1$ and  $\forall d\geq 0$, it is clear.
			
			Assume that $\forall d \geq 0$, the elements in the matrix $J_{n-1,d}$ are 0 or 1, each row has at least 1 and at most $n-1$ non-zero elements, and each column has exactly $n-1$ non-zero elements. From (\ref{equ: mac repr Jnd}), we find that the identity matrices on the diagonal add exactly one non-zero element to each row and each column of $J_{n,d}$. Thus, each row of $J_{n,d}$ has at least 1 and at most $n$ non-zero elements, and each column has exactly $n$ non-zero elements.
			We thus complete the proof of \cref{prop: basic of pro}.
		\end{proof}

		We call (\ref{equ: mac repr Jnd}) the \textbf{Macaulay representation} of the prolongation matrix $J_{n,d}$. For simplicity, the subscripts of the identity matrices on the diagonal are omitted, and their sizes are the same as the binomial coefficients in the sum on the right-hand side of the equation
		$$\binom{n+d-1}{d}=1+\binom{n-1}{1}+\binom{n}{2}+ \cdots + \binom{n+d-2}{d}.$$
		Note that the above equality correspongds to the \textbf{$d$-th Macaulay representation} of the integer $\binom{n+d-1}{d}-1$.
		
		Group the aforementioned monomial basis of $P_d$ according to the maximum power of $x_1$ they contain,  there is the following direct sum decomposition:
		$$P_d= \bigoplus_{0 \leq j \leq d}x_1^{d-j} (P/x_1)_j.$$
		This gives a   direct sum decomposition of the coordinate vector $\bm{h}$ of $$A(x)=\sum_{0 \leq j \leq d}x_1^{d-j}A_j(x_2,\cdots,x_n)$$ as follows: 
		\[
		h=\left(
		\aligned
		&\bm h_0  \\
		&\bm h_1\\
		&\vdots\\
		&\bm h_d
		\endaligned
		\right),
		\]
		where $\bm{h}_j$ is the coordinate vector of the  polynomial term $A_j$ with respect to the basis $\mathfrak{X}_{j}^{n-1}$.
		\begin{proposition}
			\label{prop: nonneg crit}
			The non-negativity of the coordinate vector $J_{n,d}\bm{h}$ of $A(x)S_1(x)$ is equivalent to the following $d+1$ linear constraints:
			\begin{align*}
				\begin{cases}
					\bm{h}_0 \geq 0, \\
					J_{n-1,i-1}\ \bm{h}_{i-1} + \bm{h}_{i} \geq 0, & 1 \leq i \leq d.\\
					J_{n-1,d} \ \bm{h}_d \geq 0.
				\end{cases}
			\end{align*}
			
		\end{proposition}
		
		\subsection{A reduction of the problem}
		In this section, we reduce the study of the weak sum-of-squares conjecture (\cref{conj: weak sos}) to the problem of estimating the lower bounds of the ranks of certain specific vectors.
		
		For the Euclidean space $\R^n$ of dimension $n$, we define the counting function
		\begin{align*}
			(P,N,Z):\R^n &\rightarrow \Z^3\\
			\bm x&\mapsto (P(\bm x), N(\bm x), Z(\bm x)).
		\end{align*}
		Here, $P(\bm{x})$, $N(\bm{x})$, and $Z(\bm{x})$ denote the numbers of positive components, negative components, and zero components of the vector $\bm{x}$, respectively. By definition, $R(\bm{x}) + Z(\bm{x}) = n$, where $R(\bm{x}) := P(\bm{x}) + N(\bm{x})$ is called the rank of the vector $\bm{x}$.
		
		The following properties of the counting function are easy to prove, and we usually do not explicitly mention them when using them:
		\begin{itemize}
			\item [(a)] $\forall \bm{x}, \bm{y} \in \R^n$, $P(\bm{x} + \bm{y}) \leq P(\bm{x}) + P(\bm{y})$;
			
			\item [(b)] $\forall \bm{y} \in \R^n$, $\bm{x} \geq 0$, $R(\bm{x}) = P(\bm{x})$, $P(\bm{y}) \leq P(\bm{x} + \bm{y})$;
			
			\item [(c)] $\forall \bm{x}, \bm{y} \in \R^n$, $| R(\bm{x}) - R(\bm{y}) | \leq R(\bm{x} + \bm{y}) \leq R(\bm{x}) + R(\bm{y})$.
		\end{itemize}
		Now that the rank of the vector has been defined, we proceed as follows:
		let $ H = \mathrm{diag} \{ h_1, \cdots, h_N\}$
		be  the matrix associated to the  diagonal Hermitian $d$-form $A(z, \bar{z}) = \sum_{|\alpha|= d} a_\alpha |z^\alpha|^2$ with respect to the left lexicographic basis   $\mathfrak{Z}_d=(z_1^d, z_1^{d-1}z_2,\cdots, z_n^d)$, where $N=\binom{n + d - 1}{d}$.
		
		Making the variable substitution $x^\alpha = |z^\alpha|^2$, we have $||z||^2 = x_1 + \cdots + x_n$, where $x_i \in \R$.
		Then  $H$ corresponds to the column vector $\bm{h}$ of dimension $\binom{n + d - 1}{d}$, which satisfies 
		\[A(z,\bar z)=\mathfrak{Z}_dH\mathfrak{Z}_d^*=\mathfrak{X}_d\bm h.\] Thus $A \in \Sigma_{n,d}$ is equivalent to $\bm{h} \geq 0$. 
		According to the definition of the prolongation map,
		\[A(z, \bar{z})||z||^2 = \mathfrak{Z}_dH\mathfrak{Z}_d^*||z||^2 = A(x)S_1(x) = \mathfrak{X}_d\bm{h}(x_1 + \cdots + x_n) = \mathfrak{X}_{d + 1}J_{n,d}\bm{h}.\]
		The above equation shows that the rank $R(A(z, \bar{z})||z||^2)$ is equal to 
		$|\mathrm{supp}(\mathfrak{X}_{d + 1}J_{n,d}\bm{h})|$,
		which  is equal to   $R(J_{n,d}\bm{h})$   the rank  of the vector $J_{n,d}\bm{h}$.

		Note that $A(z, \bar{z})||z||^2 \in \Sigma_{n,d + 1}$ is equivalent to $J_{n,d}\bm{h} \geq 0$. 
		
		Define
		$$R_{n,d} := \mathrm{min}\{ R(J_{n,d}\bm{h}) \mid \bm{h} \ngeq 0, J_{n,d}\bm{h} \geq 0, \bm{h} \in \R^{\binom{n + d - 1}{d}}\}.$$
		
		Now \cref{thm: main 1} can be restated in terms of $R_{n,d}$ as follows.
		
		\begin{theorem}\label{thm:main-1-restate}The following estimates hold:
			\begin{itemize}
				\item[(a)] $\forall n, d \geq 2$,
				\begin{align*}
					R_{n,d} \geq 3n - 4.
				\end{align*}
				\item[(b)] $\forall n \geq 6$,
				\begin{align*}
					R_{n,2} \geq \frac{n^2 + n}{2} - 6.
				\end{align*}
			\end{itemize}
		\end{theorem}
		
		In order to get the lower bound of $R_{n,d}$,  we will be dedicated to estimating $R(J_{n,d}\bm{h})$ from now.
		
		When $\bm h\geq 0$, the following theorem is a weak version of a result due to Grundmeier and Halfpap Kacmarcik.
		
		\begin{theorem}[{\cite[Proposition 3]{G3}}]\label{thm: GHP}
			For $n, d \geq 2$, let $\bm{h}$ be a non-negative vector in $\R^{\binom{n + d - 1}{d}}$ with rank $k$. Then
			$$nk - \frac{k(k - 1)}{2} \leq R(J_{n,d}\bm{h}) \leq nk, \quad k \leq n - 1,$$
			$$R(J_{n,d}\bm{h}) \geq \frac{n(n + 1)}{2}, \quad k \geq n.$$
		\end{theorem}
		
		More generally, Gao and Ng have the following estimate for the rank of $A(z, \bar{z})\|z\|_{r,s}^2$, which is not necessarily a sum of squares.
		
		\begin{theorem}[{\cite[Proposition 2.3]{G1}}]\label{thm: Gao-Ng}
			Let $n\geq 1$, $z\in \C^n$ and $A(z,\bar z)$ is a nonzero bihomogeneous Hermitian polynormial. Then 
			$$R(A(z, \bar{z})||z||^2_{r,s}) \geq r + s,\ \forall A(z, \bar{z}) \in P_{n,d}.$$
			Here, $||z||_{r,s}^2 = |z_1|^2 + \cdots + |z_r|^2 - |z_{r+1}|^2 - \cdots - |z_{r+s}|^2$, where $1 \leq r + s \leq n$.
		\end{theorem}
		
		According to the Macaulay representation of the prolongation matrix $J_{n,d}$ and the direct sum decomposition of the coordinate vector $\bm{h}$, we have 
		\begin{align}\label{equ: jndh formula}
			R(J_{n,d}\bm{h}) = R(\bm{h}_0) + \sum_{i = 1}^d R(J_{n - 1, i - 1}\bm{h}_{i - 1} + \bm{h}_i) + R(J_{n - 1, d}\bm{h}_d).
		\end{align}
		Set  $\bm{\gamma}_0 := \bm{h}_0$, $\bm{\gamma}_i := J_{n - 1, i - 1}\bm{h}_{i - 1} + \bm{h}_i, 1\leq i\leq d$, and 
		\[
		\bm \gamma :=\left(
		\aligned
		&\bm \gamma_0  \\
		&\bm \gamma_1\\
		&\vdots \\
		&\bm \gamma_d
		\endaligned
		\right),
		\]
		Then $J_{n,d}\bm h\geq 0$ is equivalent to $\bm \gamma_i\geq 0, 0\leq i\leq d$ and $J_{n-1,d}\bm h_d\geq 0$. From (\ref{equ: jndh formula}), we obtain that 
		\begin{align}\label{equ: jnd in gamma}
			R(J_{n,d}\bm{h}) = R(\bm{\gamma}) + R(J_{n - 1, d}\bm{h}_d) = R(\bm{\gamma}) + R\left(\sum_{i = 0}^d (-1)^{d - i}J_{n - 1, d}\cdots J_{n - 1, i}\bm{\gamma}_i\right).
		\end{align}
		Note that $R(\bm{\gamma})>0$. Otherwise, $\bm{h}=0$, this is a trivial case. 
		
		Estimating $R(J_{n,d}\bm{h})$ presents a challenge due to the difficulty in estimating the rank of the alternating sum on the right-hand side of the equation (\ref{equ: jnd in gamma}).  In the following, we give estimates of $R(J_{n,d}\bm{h})$ when $\bm{\gamma}$ is simple, this is crucial to the proof of Theorem \ref{thm:main-1-restate} (a), but not necessary for proving Theorem   \ref{thm:main-1-restate} (b). Since the proof of Theorem \ref{thm:main-1-restate} (a) relies on Theorem \ref{thm:main-1-restate} (b), we first prove the lemmas that will be used in  the latter’s proof. They are essentially the same as Macaulay's estimates, despite their apparent differences.
		
		\section{Macaulay type estimate of counting Functions}

		In the following lemma, we give an estimate of the counting function on the \textbf{first prolongation} of  vectors in $\mathbb{R}^{n}$.
		
		\begin{lemma}\label{lem: counting ineq}
			Let  $\bm{a}$  be a vector in $\mathbb{R}^n$. Denote $P(\bm a),N(\bm a),Z(\bm a)$ by $P, N, Z$ respectively. Then We have
			\begin{align*}
				P(J_{n,1}\bm{a}) &\geq \frac{P(P + 1)}{2} + PZ,\\
				N(J_{n,1}\bm{a}) &\geq \frac{N(N + 1)}{2} + NZ,\\
				Z(J_{n,1}\bm{a}) &\leq \frac{Z(Z + 1)}{2} + PN.
			\end{align*}
			The equalities hold when $PN = 0$. 
		\end{lemma}
		
		\begin{proof}
			
			The vector $J_{n,1}\bm{a}$ is the coordinate vector of the polynomial $(a_1x_1 + \cdots + a_nx_n)(x_1 + \cdots + x_n)$ with respect to the basis $\mathfrak{X}_2$. Due to the symmetry of $S_1$, we may assume, without loss of generality, that the first $P$ components of $\bm{a}$ are positive and the last $Z$ components are zero. This arrangement ensures that the coefficients of $x_ix_j$ are positive when $1 \leq i \leq j \leq P$, and the coefficients of $x_ix_{P + N + j}$ are also positive when $1 \leq i \leq P$ and $1 \leq j \leq Z$. Therefore, we have
			\[
			P(J_{n,1}\bm{a}) \geq \frac{P(P + 1)}{2} + PZ.
			\]
			When $N = 0$, $N(J_{n,1}\bm{a}) = 0$, and there are no other non-zero terms besides those listed above. Thus, the inequality is sharp.
			
			Similarly, we can prove that 
			\[N(J_{n,1}\bm{a}) \geq \frac{N(N + 1)}{2} + NZ.\]
			
			Finally, noting that $n = P + N + Z$, we have
			\[
			P(J_{n,1}\bm{a}) + N(J_{n,1}\bm{a}) + Z(J_{n,1}\bm{a}) = \frac{n(n + 1)}{2} = \frac{P(P + 1)}{2} + PZ + \frac{N(N + 1)}{2} + NZ + \frac{Z(Z + 1)}{2} + PN.
			\]
			Subtracting $P(J_{n,1}\bm{a}) + N(J_{n,1}\bm{a})$ from both sides and combining the inequalities above, we get the desired estimate of $Z(J_{n,1}\bm{a})$. This completes the proof of \cref{lem: counting ineq}.
		\end{proof}

		We can also estimate the counting of zeroes of the \textbf{second prolongation} of vectors in $\mathbb{R}^{n}$. This is a \textbf{Macaulay type} estimate.
		
		\begin{lemma}\label{lem: est sec pro}
			For $n \geq 2$, and any $\bm{a} \in \mathbb{R}^n$ such that $N(\bm{a}) \geq 2$, there exists an integer sequence $\{c(n)\}$ such that
			\begin{align*}
				Z(J_{n,2}J_{n,1}\bm{a}) \leq \binom{n + 1}{3} - nN(\bm{a}) + n + c(n),
			\end{align*}
			where $c(2) = 1$, $c(3) = 2$, $c(4) = 4$, and $c(n) = 10 - n$ for $n \geq 5$.
		\end{lemma}
		
		\begin{proof}
			Define
			$$L(n, N(\bm a)) := \binom{n + 1}{3} - nN(\bm a) + n.$$
			
			We only need to verify that  $c(n) \geq Z(J_{n,2}J_{n,1}\bm{a}) - L(n, N(\bm{a}))$.
			
			By definition,  $Z(J_{n,2}J_{n,1}\bm{a})$ is the number of zero coefficients of the homogeneous polynomial $(a_1x_1 + \cdots + a_nx_n)(x_1 + \cdots + x_n)^2$ with respect to the basis $\mathfrak{X}_3$.  
			Note that the counting function is invariant under coordinate permutations.
			Without loss of generality, assume the components of $\bm{a}$ are arranged in descending order. According to the Macaulay representation of $J_{n,2}$ and $J_{n,1}$ (see (\ref{equ: mac repr Jnd})), their composition is
			\begin{align*}
				J_{n,2}J_{n,1} = \begin{pmatrix}
					1 & 0 \\
					2J_{n - 1,0} & I \\
					J_{n - 1,1}J_{n - 1,0} & 2J_{n - 1,1} \\
					0 & J_{n - 1,2}J_{n - 1,1}
				\end{pmatrix}.
			\end{align*}
			Then $J_{n,2}J_{n,1}\bm{a}$ is
			\begin{align}\label{equ: 2nd pro}
				\begin{pmatrix}
					1 & 0 \\
					2J_{n - 1,0} & I \\
					J_{n - 1,1}J_{n - 1,0} & 2J_{n - 1,1} \\
					0 & J_{n - 1,2}J_{n - 1,1}
				\end{pmatrix}
				\begin{pmatrix}
					\bm{a}_0 \\
					\bm{a}_1
				\end{pmatrix}
				=
				\begin{pmatrix}
					\bm{a}_0 \\
					2J_{n - 1,0}\bm{a}_0 + \bm{a}_1 \\
					J_{n - 1,1}J_{n - 1,0}\bm{a}_0 + 2J_{n - 1,1}\bm{a}_1 \\
					J_{n - 1,2}J_{n-1,1}\bm{a}_1
				\end{pmatrix}.
			\end{align}
			
			Observe that after arranging the components of $\bm a$ in descending order, $\bm{a}_0 = a_0$ is the largest component of $\bm{a}$.

			We discuss cases based on $P(\bm{a})$.
			
			\textbf{Case I}: $N(\bm{a}) \geq 2$, $P(\bm{a}) = 0$.
			
			Denote the monomial subspace of $P_1$ generated by the monomials corresponding to the last $N(\bm{a})$ coefficients of $\mathfrak{X}_1\bm{a}$ as $A_1$, so $A_1^{ \langle 1 \rangle}$ is a $P_2$-monomial subspace. In this case, $Z:=Z(\bm{a}) = \mathrm{codim}(A_1)$. By the Macaulay estimate (\cref{thm: Mac estimate}),
			\begin{align*}
				Z(J_{n,2}J_{n,1}\bm{a}) = \mathrm{codim}((A_1^{\langle 1 \rangle})^{\langle 1 \rangle}) \leq (\mathrm{codim}(A_1^{\langle 1 \rangle}))^{\langle 2 \rangle} \leq ((\mathrm{codim}(A_1))^{\langle 1 \rangle})^{\langle 2 \rangle} = ((Z)^{\langle 1 \rangle})^{\langle 2 \rangle}.
			\end{align*}
			Since the $1$-Macaulay representation of $Z$ is $Z = \binom{Z}{1}$, so $((Z)^{\langle 1 \rangle})^{\langle 2 \rangle} = (\binom{Z + 1}{2})^{\langle 2 \rangle} = \binom{Z + 2}{3}$. 
			
			When $n=2$, it is direct to check that $Z(J_{n,2}J_{n,1}\bm{a})=0$, and $c(2)=1$ satisfies \cref{lem: est sec pro}.
			
			Fixing $n\geq 3$, consider the function of the variable $N:=N(\bm a)$:
			\begin{align*}
				g(N):=Z(J_{n,2}J_{n,1}\bm{a})-L(n,N)=	\binom{n - N + 2}{3} - L(n, N).
			\end{align*}
			We compute that 
			\begin{align*}
				\Delta g(N) := g(N+1)-g(N)=n - \binom{n - N + 1}{2} ,\quad 2 \leq N \leq n-1.
			\end{align*}
			It is easy to see that $ \Delta g(N) $ increases as $N$ increases, which  means that the maximum value of $g(N)$ must be achieved at endpoints $N = 2$ or $N = n$, i.e., 
			\begin{align*}
				g(N) \leq  \mathrm{max} \left\{ n - \binom{n}{2}, \binom{n}{2} - \binom{n}{3}\right \}, \quad n \geq 3.
			\end{align*}
			Then it is easy to check the following
			\begin{align*}
				c(n) \geq \mathrm{max} \left\{ n - \binom{n}{2}, \binom{n}{2} - \binom{n}{3}\right \}, \quad n \geq 3.
			\end{align*}
			This completes the proof of the \textbf{Case I}.
			
			We are on the way to  prove \cref{lem: est sec pro} by induction argument.
			
			When $n=2$, \textbf{Case I} already proved the \cref{lem: est sec pro}.
			
			Now, assume that for any $\bm{a}' \in \R^{n - 1}$ with $N(\bm{a}') \geq 2$ for   $n\geq 3$,
			\begin{align*}
				Z(J_{n - 1,2}J_{n - 1,1}\bm{a}') - L(n - 1, N(\bm{a}')) \leq c(n - 1).
			\end{align*}
			
			Based on \textbf{Case I}, it suffices to prove \cref{lem: est sec pro} for the remaining case as follows.

			\textbf{Case II: $N(\bm{a}) \geq 2$, $P(\bm{a}) \geq 1$ (this case occurs only when $n \geq 3$)}.
			
			Without loss of generality, we assume that  $\bm{a}_0 = 1$. 
			
			Define $\bm{b}=2J_{n - 1,0}\bm a_0 +\bm{a}_1$ and $\bm{c} = J_{n - 1,0}\bm a_0 + 2\bm{a}_1$. Since
			$
			\begin{pmatrix}
				2 & 1 \\
				1 & 2
			\end{pmatrix}
			$
			is an invertible matrix, the elements of the vectors $\bm{b}$ and $\bm{c}$ cannot be zero simultaneously at the same position. Therefore, $Z(\bm{b}) + Z(\bm{c}) \leq n - 1$. From (\ref{equ: 2nd pro}) and $\bm a_0=1$, we deduce that 
			\begin{align}\label{equ: 2nd pro zero ineq}
				Z(J_{n,2}J_{n,1}\bm{a}) \leq n - 1 - Z(\bm{c}) + Z(J_{n - 1,1}\bm{c}) + Z(J_{n - 1,2}J_{n - 1,1}\bm{a}_1).
			\end{align}
			
			Subdivide \textbf{Case II} into the following two sub-cases and verify the lemma for each:
			\begin{itemize}
				\item  \textbf{Case II-1: $N(\bm{a}) \geq 2$, $P(\bm{a}) \geq 1$, and  $N(\bm{a}_1) = n - 1$};
				\item \textbf{Case II-2: $ N(\bm{a}) \geq 2$, $P(\bm{a}) \geq 1$, and $N(\bm{a}_1)\leq  n - 2$}.
			\end{itemize}
			
			
			\textbf{Case II-1: $N(\bm{a}) \geq 2$, $P(\bm{a}) \geq 1$, and  $N(\bm{a}_1) = n - 1$}. 
			
			From  \cref{prop: basic of pro}, we know that $J_{n,d}$ has no negative elements and is non-zero in any row. Therefore, when $\bm{h} < 0$, we must have $J_{n,d}\bm{h} < 0$. So when $N(\bm{a}_1) = n - 1$, i.e., $\bm a_1<0$, we have  $Z(J_{n - 1,2}J_{n - 1,1}\bm{a}_1) = 0$. According to the third inequality in  \cref{lem: counting ineq}, 
			\[-Z(\bm{c}) + Z(J_{n - 1,1}\bm{c}) \leq Z(\bm{c})(Z(\bm{c}) - 1)/2 + N(\bm{c})P(\bm{c}).\] 
			Plugging the above estimate into (\ref{equ: 2nd pro zero ineq}) and using the Cauchy inequality for the product $N(\bm{c})P(\bm{c})$, we get that 
			\begin{align}\label{equ: 2nd pro est zero ineq-2}
				Z(J_{n,2}J_{n,1}\bm{a}) \leq n - 1 + \frac{Z(\bm{c})(Z(\bm{c}) - 1)}{2} + \frac{(n - 1 - Z(\bm{c}))^2}{4} = n - 1 + \frac{3Z(\bm{c})^2 - 2nZ(\bm{c}) + (n - 1)^2}{4}.
			\end{align}
			The function on $Z(\bm c)$ on the right hand side attains  the maximum value  at the left or right endpoint, $Z(\bm{c}) = 0$ or $Z(\bm{c}) = n - 1$. Direct computation shows   that the value at the left endpoint is $n - 1 + \frac{(n - 1)^2}{4}$, which is less than or equal to the value $\frac{n(n - 1)}{2}$ at the right endpoint.
			
			It is not difficult to verify that
			\begin{align*}
				c(n) \geq \binom{n}{2} - L(n, n - 1) \geq Z(J_{n,2}J_{n,1}\bm{a}) - L(n, n - 1), \quad n \geq 3.
			\end{align*}
			This completes the proof of \cref{lem: est sec pro} for \textbf{Case II-1}.
			
			\textbf{Case II-2: $ N(\bm{a}) \geq 2$, $P(\bm{a}) \geq 1$, and $N(\bm{a}_1)\leq  n - 2$, (this case occurs only when $n \geq 4$).}
			
			In this case,  $2\leq N:=N(\bm a)=N(\bm a_1) \leq n - 2$, $0\leq Z(\bm b), Z(\bm c)\leq n-2$, by similar discussions as in the beginning of \textbf{Case II}, we can see that $Z(\bm b)+Z(\bm c)\leq n-2$, i.e., $Z(\bm b)\leq n-2-Z(\bm c)$.
			
			We compute that 
			\begin{align}
				L(n, N) - L(n - 1, N) = \binom{n}{2} - N + 1 \geq  \binom{n-1}{2}+ 2.
				\label{eq5-6}
			\end{align}
			From (\ref{equ: 2nd pro zero ineq}) and the \textbf{inductive hypothesis}, similar with the computation in (\ref{equ: 2nd pro est zero ineq-2}), we obtain that 
			\begin{align*}
				Z(J_{n,2}J_{n,1}\bm{a}) &\leq n - 2 +\frac{3Z(\bm{c})^2 - 2nZ(\bm{c}) + (n - 1)^2}{4} +Z(J_{n - 1,2}J_{n - 1,1}\bm{a}_1)\\
				&\leq n-2+\mathrm{max} \left\{ \frac{(n - 1)^2}{4}, 1/4 + \frac{(n - 2)(n - 3)}{2} \right\} + L(n - 1, N(\bm{a}_1)) + c(n - 1).
			\end{align*}
			Since $(n - 1)^2/4 \leq 1+\delta_4^n + (n - 2)(n - 3)/2$, $\delta_4^n$ is the Kronecker delta, by noting the existence of (\ref{eq5-6}), we deduce that
			\[
			Z(J_{n,2}J_{n,1}\bm{a}) - L(n, N(\bm{a})) \leq n - 2+1+\delta_4^n + \binom{n - 2}{2} - \binom{n - 1}{2} - 2 + c(n - 1)   = c(n - 1) - 1+ \delta_4^n \leq c(n).
			\]
			This completes the proof of \textbf{Case II-2}.
			
			All cases have been verified, and the proof of \cref{lem: est sec pro} is complete.
		\end{proof}
		\section{Proof of \cref{thm: main 1}(B)}
		
		To prove \cref{thm: main 1} (B) (also Theroem \ref{thm:main-1-restate} (b)), it suffices to prove the case where $\bm{h}_2 \ngeq 0$. 
		In fact , since $R(A(x)S_1(x))$ is invariant under the permutation of the indices of $x_i$,  in \cref{thm: main 1} (B), where $2 = d < n$, without loss of generality, we assume that there exists a monomial with a negative coefficient does not contain $x_1$, i.e. $\bm{h}_2 \ngeq 0$.
		
		\begin{lemma}\label{lem: 3rank big 1}
			Let $n \geq 3$. Assume that  $J_{n,2}\bm{h} \geq 0$ and $\bm{h}_2 \ngeq 0$.  Then  we have the following estimate:
			$$P(\bm{h}_1),\ R(J_{n-1,0}\bm{h}_0 + \bm{h}_1),\ R(J_{n-1,1}\bm{h}_1 + \bm{h}_2) \geq 1.$$
		\end{lemma}
		\begin{proof}
			The first statement is easy to prove by contradiction. Suppose $P(\bm{h}_1) = 0$. Then $\bm{h}_2 \geq -J_{n-1,1}\bm{h}_1 \geq 0$, which contradicts $\bm{h}_2 \ngeq 0$. Therefore, $P(\bm{h}_1) \geq 1$.
			
			Since $J_{n,2}\bm{h} \geq 0$, we have $R(J_{n-1,0}\bm{h}_0 + \bm{h}_1) = P(J_{n-1,0}\bm{h}_0 + \bm{h}_1)$. Also, because $\bm{h}_0 \geq 0$, we get $P(J_{n-1,0}\bm{h}_0 + \bm{h}_1) \geq P(\bm{h}_1) \geq 1$.
			
			If $R(J_{n-1,1}\bm{h}_1 + \bm{h}_2) = 0$, then $-J_{n,2}J_{n,1}\bm{h}_1 \geq 0$. We claim that $\bm{h}_1 \leq 0$ must hold in this case. We prove this claim by induction. For $n = 1$, it is obvious that $J_{1,2}J_{1,1}x \leq 0$ implies $x \leq 0$. Assume for any $\bm{x}' \in \R^{n-1}$, if $J_{n,2}J_{n,1}\bm{x}' \leq 0$, then $\bm{x}' \leq 0.$ Now let $\bm x\in \R^n$, denote $\bm{x}^t = (x_1, (\bm{x}')^t)^t \in \R^n$. From $J_{n,2}J_{n,1}\bm{x} \leq 0$ and (\ref{equ: 2nd pro}), we can deduce that $x_1 \leq 0$ and $J_{n-1,2}J_{n-1,1}\bm{x}' \leq 0$. By the inductive hypothesis, $\bm{x}' \leq 0$, so $\bm{x} \leq 0$.  Thus, the claim follows.
			
			Now, from $-J_{n,2}J_{n,1}\bm{h}_1 \geq 0$, we have $J_{n,2}J_{n,1}\bm{h}_1 \leq 0$. The above claim shows that $\bm{h}_1 \leq 0$, which contradicts $P(\bm{h}_1) \geq 1$. The proof of \cref{lem: 3rank big 1} is complete.
		\end{proof}
		
		\begin{corollary}\label{cor: est of rjn2h}
			Let $n \geq 3$. Assume that  $J_{n,2}\bm{h} \geq 0$ and $\bm{h}_2 \ngeq 0$. Then 
			\begin{align*}
				R(J_{n,2}\bm{h}) \geq 3 + R_{n-1,2}.
			\end{align*}
		\end{corollary}
		
		\begin{proof}
			We only need to consider two cases: $\bm{h}_0 = 0$ or $\bm{h}_0 \neq 0$.
			
			When $P(\bm{h}_0) = 0$, from $\bm{\gamma}_1 = J_{n-1,0}\bm{h}_0 + \bm{h}_1 \geq 0$, we know that $\bm{h}_1 \geq 0$. Since $\bm{h}_2$ has negative components, $\bm{h}_1$ cannot be zero. Also, because $J_{n-1,2}\bm{h}_2 \geq 0$,   \cite[Theorem 3.1(i)]{H1} tells us that $P(\bm{h}_2) \geq n - 1$. Therefore, we obtain that
			\begin{align*}
				R(J_{n,2}\bm{h}) \geq 0 + 1 + (n - 1) + R(J_{n-1,2}\bm{h}_2) = n + R_{n-1,2}.
			\end{align*}
			
			When $P(\bm{h}_0) = 1$, using Lemma \ref{lem: 3rank big 1}, we can conservatively estimate that
			$$R(J_{n,2}\bm{h}) \geq 1 + 1 + 1 + R(J_{n-1,2}\bm{h}_2) = 3 + R_{n-1,2}.$$
			This complete the proof of \cref{cor: est of rjn2h}.
		\end{proof}
		
		To prove part (B) of Theorem \cref{thm: main 1}, we first need to estimate $R_{2,2}$, $R_{3,2}$, $R_{4,2}$, and $R_{5,2}$. Now consider the case when $n = 2$. By taking the homogeneous polynomial in \cref{thm: Gao-Ng}  to be diagonal and the norm to be the standard Euclidean norm, we immediately get $R_{2,d} \geq 2$ for $d \geq 2$.

		Combining \cref{thm: Gao-Ng} with Corollary \ref{cor: est of rjn2h}, we have
		\begin{align*}
			R_{3,2} &\geq 3 + R_{2,2} \geq 5, \\
			R_{4,2} &\geq 3 + R_{3,2} \geq 8, \\
			R_{5,2} &\geq 3 + R_{4,2} \geq 11.
		\end{align*}
		
		This proves \cref{thm: main 1} (A) for $n = 3, 4, 5$ and $d = 2$. Moreover, the estimate of $R_{3,2}$ and $R_{4,2}$ are sharp, and their minimum values are achieved at the coordinate vectors of the following polynomials:
		\begin{align*}
			f &= \frac{1}{2}(x_1 - x_2 + x_3)^2 + \frac{1}{2}x_2^2 + \frac{1}{2}(x_1 + x_3)^2, \\
			g &= \frac{1}{2}(x_1 - x_2 - x_3 + x_4)^2 + \frac{1}{2}(x_1 + x_4)^2 + \frac{1}{2}(x_2 + x_3)^2.
		\end{align*}

		From Lemma \ref{lem: 3rank big 1}, we know $R(\bm{\gamma}_2) \geq 1$.
		For $n \geq 6$, we divide the case  where $J_{n,2}\bm{h} \geq 0$ and $\bm{h}_2 \ngeq 0$ into the following  three subclasses: 
		\begin{itemize}
			\item \textbf{Case I}: $\bm{h}_0 = 0$;
			\item \textbf{Case II}: $\bm{h}_0 = 1$ (if $\bm h_0\neq 0$, then $\bm h_0>0$, we can always assume $\bm h_0=1$ upto scaling by a positive constant) and either $N(\bm{h}_1) \leq 1$ or $R(\bm{\gamma}_2) \geq N(\bm{h}_1)$ holds;
			\item \textbf{Case III}: $\bm{h}_0 = 1$, $N(\bm{h}_1) \geq 2$, and $R(\bm{\gamma}_2) \leq N(\bm{h}_1) - 1$.
			
		\end{itemize}

		First, for the first two classes, we \textbf{claim} that the following inequality holds:
		\begin{align}\label{equ:}
			R(J_{n,2}\bm{h}) \geq n + R_{n-1,2}.
		\end{align}
		
		For \textbf{Case I}, the proof of the above inequality is exactly the same as the corresponding proof  in \cref{cor: est of rjn2h}. For  \textbf{Case II}, since $R(\bm{\gamma}_2) \geq 1$, we get that 
		\begin{align*}
			R(J_{n,2}\bm{h}) &\geq 1 + (n - 1 - N(\bm{h}_1)) + R(\bm{\gamma}_2) + R(J_{n-1,2}\bm{h}_2) \\
			&\geq n + R(J_{n-1,2}\bm{h}_2) \geq n + R_{n-1,2}.
		\end{align*}
		The claim follows.

		For \textbf{Case III}, we have the following estimate:
		\begin{align*}
			R(J_{n,2}\bm{h}) &\geq n - N(\bm{h}_1) + R(\bm{\gamma}_2) + R(J_{n-1,2}\bm{h}_2) \\
			&\geq n - N(\bm{h}_1) + R(\bm{\gamma}_2) + N(J_{n-1,2}J_{n-1,1}\bm{h}_1) \\
			&= n - N(\bm{h}_1) + R(\bm{\gamma}_2) + \binom{n + 1}{3} - P(J_{n-1,2}J_{n-1,1}\bm{h}_1) - Z(J_{n-1,2}J_{n-1,1}\bm{h}_1).
		\end{align*}
		
		The second inequality holds because $R(J_{n-1,2}\bm{h}_2) = P(J_{n-1,2}(\bm{\gamma}_2 - J_{n-1,1}\bm{h}_1)) \geq N(J_{n-1,2}J_{n-1,1}\bm{h}_1)$ and $\bm \gamma_2\geq 0$.
		
		Now $\bm{h}_1$ satisfies the conditions of Lemma \ref{lem: est sec pro}, so $Z(J_{n-1,2}J_{n-1,1}\bm{h}_1) \leq \binom{n}{3} - (n - 1) N(\bm{h}_1) + n - 1 + c(n - 1)$. Since $J_{n-1,2}J_{n-1,1}\bm{h}_1 \leq J_{n-1,2}\bm{\gamma}_2$, we have $P(J_{n-1,2}J_{n-1,1}\bm{h}_1) \leq P(J_{n-1,2}\bm{\gamma}_2)$.
		
		Because $\bm{\gamma}_2 \geq 0$, 
		according to \cref{thm: GHP}, $R(J_{n-1,2}\bm{\gamma}_2) \leq (n - 1)R(\bm{\gamma}_2) \leq (n - 1)( N(\bm{h}_1) - 1)$. Therefore,
		\begin{align*}
			R(J_{n,2}\bm{h}) &\geq n -  N(\bm{h}_1) + R(\bm{\gamma}_2) + \binom{n + 1}{3} - (n - 1)R(\bm{\gamma}_2) \\
			&\quad - \binom{n}{3} - (n - 1) +  N(\bm{h}_1)(n - 1) - c(n - 1) \\
			&\geq \binom{n}{2} - (10 - (n - 1)) + (n-2) N(\bm{h}_1) + 1 - (n - 2)( N(\bm{h}_1) - 1) \\
			&= \frac{n^2 + 3n}{2} - 12.
		\end{align*}
		In the last inequality, we used $c(n - 1) = 10 - (n - 1)$, which is correct for $n \geq 6$ according to  \cref{lem: est sec pro}. This gives the estimate for \textbf{Case III}. 
		
		Combining all the estimates for the three cases, we obtain that 
		\begin{equation}
			R(J_{n,2}\bm{h}) \geq \min\left\{n + R_{n-1,2}, \frac{n^2 + 3n}{2} - 12\right\}, \quad n \geq 6.
		\end{equation}
		
		Note that
		$$\frac{n^2 + 3n}{2} - 12 \geq \frac{n^2 + n}{2} - 6, \quad n \geq 6.$$
		
		Taking $R_{5,2} \geq 11 \geq (5^2 + 5)/2 - 6$ as the starting point for induction, it is easy to show that
		\begin{align*}
			R_{n,2} \geq \frac{n^2 + n}{2} - 6, \quad n \geq 6.
		\end{align*}
		\textbf{This completes the proof of \cref{thm: main 1} (B).}
		
		A simple calculation shows that the conjectured lower bound (\ref{eq1-2}) is of the order $O(n^{3/2})$, while the above result is of the order $O(n^2)$.
		
		\section{Proof of \cref{thm: main 1}(A)}
		In this section, we will present the proof of  \cref{thm: main 1}(A).
		It can be observed that when $d = 2$, the lower bound obtained in \cref{thm: main 1}(B) is not only numerically but also asymptotically larger than that in \cref{thm: main 1}(A). Unfortunately, when $d \geq 3$, it is extremely challenging to replicate the proof from the previous section. Brooks and Grundmeier \cite{B1} were able to provide an estimate when $N(\bm{h}) = 1$.
		
		\begin{proposition}[{\cite[Proposition 2]{B1}}]
			Suppose $N(\bm{h}) = 1$ and $J_{n,d}\bm{h} \geq 0$. Then $R(J_{n,d}\bm{h}) \geq \frac{n(n+1)}{2} - 1$.
		\end{proposition}
		
		The remaining discussion is based on the condition $N(\bm{h}) \geq 2$.
		
		When $N(\bm{h}) = 2$, Brooks and Grundmeier insightfully used graded Betti numbers to characterize the cancellation relationship between positive and negative coefficients. However, due to computational difficulties in higher dimensions, their discussion was limited to $\C^3$. As the dimension and degree increase, many challenges arise. For example, when $d \geq n$, there exist cases where $N(\bm{h}_d) = 0$ regardless of how the coordinates are rotated, this poses difficulties for  using  induction. Therefore, we present the following illuminating proposition to help readers understand how the lower bound in \cref{thm: main 1} (A) is derived.
		
		\begin{proposition}
			When $n \geq 4$, if $\sum_{i=1}^{d-2}N(\bm{h}_i) + N(\bm{h}_d) = 0$, $P(\bm{h}_d) = N(\bm{h}_{d-1}) = 2$, and $J_{n,d}\bm{h} \geq 0$, then
			$$R(J_{n,d}\bm{h}) \geq 3n - 4.$$
			\label{prop4-1}
		\end{proposition}
		
		\begin{proof}
			From Theorem \ref{thm: Gao-Ng}, we obtain the estimate
			$$R(J_{n-1,d-1}\bm{h}_{d-1}) \geq n - 1 \geq 3 > P(\bm{h}_d).$$
			Therefore, $R(J_{n-1,d-1}\bm{h}_{d-1} + \bm{h}_d) \geq 1$, and $P(\bm{h}_{d-2}) \geq 1$. Otherwise, $\bm{h}_{d-1} \geq 0$, which contradicts the given conditions. Using the Macaulay representation of $J_{n,d}$, we have
			\begin{align*}
				&R(J_{n,d}\bm{h}) = R(\bm{h}_0) + \sum_{i=1}^d R(J_{n-1,i-1}\bm{h}_{i-1} + \bm{h}_i) + R(J_{n-1,d}\bm{h}_d) \\
				\geq & R(J_{n-1,d-3}\bm{h}_{d-3} + \bm{h}_{d-2}) +
				R(J_{n-1,d-2}\bm{h}_{d-2} + \bm{h}_{d-1}) +
				R(J_{n-1,d-1}\bm{h}_{d-1} + \bm{h}_d) + 
				R(J_{n-1,d}\bm{h}_d) \\
				\geq & R(\bm{h}_{d-2}) + R(J_{n-1,d-2}\bm{h}_{d-2} + \bm{h}_{d-1})
				+ R(J_{n-1,d-1}\bm{h}_{d-1} + \bm{h}_d) + 
				R(J_{n-1,d}\bm{h}_d) \\
				\geq & 1 + (n - 1 - 2) + 1 + 2(n - 1) - \frac{2 \times 1}{2} = 3n - 4.
			\end{align*}
			Where the second inequality use the assumption that $N(\bm h_{d-3})=0$, and the third inequality uses \cref{thm: Gao-Ng}  and  \cref{thm: GHP}.
		\end{proof}

		\subsection{Estimation of the alternating sum in simple cases}
		To remove the additional conditions in the above proposition, we need to perform more refined calculations. Beyond the crude calculations , we note that constructing an easily understandable proof for the following lemmas is quite challenging. However, we emphasize that skipping these calculations and proceeding directly to next subsection  should not hinder your understanding of our main proof.
		\begin{lemma}\label{lem: estimate on jnd for gamma rank 1}
			For $n, d \geq 2$, using the same notations as above, if $N(\bm{h}) \geq 1$ and $J_{n,d}\bm{h} \geq 0$, and the $(d + 1)$-dimensional vector $(R(\bm{\gamma}_0), \cdots, R(\bm{\gamma}_d))^t$ has only one non-zero component, then
			\begin{align*}
				R(J_{n,d}\bm{h}) \geq \binom{n+1}{3}+1.
			\end{align*}
		\end{lemma}
		
		\begin{proof}
			Suppose $R(\bm{\gamma}_a) \geq 1$ and $R(\bm{\gamma}_i)=0$ for $0\leq i\leq d, i\neq a\leq d$. 
			From (\ref{equ: jnd in gamma}), we have 
			\begin{align}\label{equ: est gamma small-1}
				R(J_{n,d}\bm{h}) \geq 1 + R((-1)^{d-a}J_{n - 1, d}\cdots J_{n-1, a}\bm{\gamma}_a).
			\end{align}
			
			We first claim  that 
			\begin{align}\label{equ: est gamma small-2}
				R(J_{n - 1, d}\cdots J_{n - 1, a}\bm{\gamma}_a)\geq |\mbox{supp}((x_2 + \cdots + x_n)^{d - a + 1})|=\binom{n+d-a-1}{d-a+1}.    
			\end{align}
			In fact, let $f(x_2, \cdots, x_n) = \mathfrak{X}_a\bm{\gamma}_a = \sum_j a_j x^{I_j}$, where $a_j > 0$, $|I_j| = a$, and $1 \leq j \leq R(\bm{\gamma}_a)$. Then 
			\[R(J_{n - 1, d}\cdots J_{n - 1, a}\bm{\gamma}_a)=|\mbox{supp}(f(x_2, \cdots, x_n)(x_2 + \cdots + x_n)^{d - a + 1})|\geq |\mbox{supp}(a_1x^{I_1}(x_2 + \cdots + x_n)^{d - a + 1})|.\] 
			This proves the claim.
			
			Moreover, we claim that $d - a \geq 2$. Combining this with  (\ref{equ: est gamma small-1}) and (\ref{equ: est gamma small-2}) completes the proof of Lemma~\ref{lem: estimate on jnd for gamma rank 1}. Actually, the inequalities $\bm{\gamma}_a \geq 0$ and $(-1)^{d - a}J_{n - 1, d}\cdots J_{n - 1, a}\bm{\gamma}_a \geq 0$ imply that $(-1)^{d - a}=1$. Consequently, $d - a$ is even. If $d = a$, we obtain $\bm{\gamma}_0=\cdots=\bm{\gamma}_{d - 1}=0$. This, in turn, implies that $\bm{h}_0=\cdots=\bm{h}_{d - 1}=0$. Thus, $\bm{h}_d\geq 0$, which contradicts the condition $N(\bm{h})\geq 1$. Therefore, we conclude that $d - a\geq 2$.
		\end{proof}
		
		\begin{lemma}\label{lem: est gamma rank 2}
			For $n, d \geq 2$, if the non-negative vector $\bm{v}$ is a linear combination of two vectors $\bm{\gamma}_a, \bm{\gamma}_b \geq 0$ of different dimensions after multiple prolongations, i.e.,
			\begin{align*}
				\bm{v} = (-1)^{d - a}J_{n,d}J_{n,d - 1}\cdots J_{n,a}\bm{\gamma}_a + (-1)^{d - b}J_{n,d}J_{n,d - 1}\cdots J_{n,b}\bm{\gamma}_b,
			\end{align*}
			where $0 \leq a < b \leq d$ and $R(\bm{\gamma}_a) = R(\bm{\gamma}_b) = 1$, then
			\begin{align*}
				R(\bm{v}) \geq \binom{n + 2}{3} - 1.
			\end{align*}
		\end{lemma}
		
		\begin{proof}
			We first show that $(-1)^{d - a} = 1$. Since $R(\bm{\gamma}_a) = R(\bm{\gamma}_b) = 1$, it is not difficult to find that 
			\[J_{n,d}J_{n,d - 1}\cdots J_{n,a}\bm{\gamma}_a=|\mbox{supp}((x_1 + \cdots + x_n)^{d - a + 1})|=\binom{n + d - a}{d - a + 1}.\] 
			Because $a < b$, we have 
			\[R(J_{n,d}J_{n,d - 1}\cdots J_{n,a}\bm{\gamma}_a) > R(J_{n,d}J_{n,d - 1}\cdots J_{n,b}\bm{\gamma}_b).\]
			If $(-1)^{d - a} = -1$, then $\bm{v}$ must have negative components, contradicting the condition.
			
			In what follows, we continue the proof by considering two cases based on the sign of $(-1)^{d - b}$.
			
			
			\textbf{Case I:  $(-1)^{d - b} = 1$.}  In this case, the second term in  $\bm \nu$ has non-negative coefficients, so we have that 
			\[R(\bm{v}) \geq \binom{n + d - a}{d - a + 1} = \binom{n + d - a}{n - 1}.\] 
			Since $d - a$ is an even number and $a < d$, we get that 
			\begin{align}\label{equ: est lambda1}R(\bm{v}) \geq \binom{n + 2}{n - 1} = \binom{n + 2}{3}.\end{align}
			
			\textbf{Case II:   $(-1)^{d - b} = -1$.} Without loss of generality,  we assume the only non-zero component in $\bm{\gamma}_a$ is 1. We have two polynomials  $x^\alpha := \mathfrak{X}_a\bm{\gamma}_a$ and $\lambda x^\beta := \mathfrak{X}_b\bm{\gamma}_b$ for some $\lambda>0$.
			According to the definition of the prolongation map,  we obtain that 
			\[P(x) := \mathfrak{X}_d\bm{v} = x^\alpha(x_1 + \cdots + x_n)^{d - a + 1} - \lambda x^\beta(x_1 + \cdots + x_n)^{d - b + 1}.\]
			Note that  $|\mbox{supp}(P(x))|=R(\bm{v})$.
			
			Let $a = d - 2l$ and $b = d - 2t + 1$, where $l \geq t \geq 1$. Direct computation yields that 
			\begin{align*}
				P(x) &= x^\alpha(x_1 + \cdots + x_n)^{2l + 1} - \lambda x^\beta(x_1 + \cdots + x_n)^{2t} \\
				&= x^\alpha\sum_{|I| = 2l + 1} \frac{(2l + 1)!}{I!}x^I - \lambda x^\beta\sum_{|J| = 2t} \frac{(2t)!}{J!}x^J.
			\end{align*}
			Here $I! := (i_1)!\cdots(i_n)!$ if $I = (i_1, \cdots, i_n)$.
			
			The discussion on \textbf{Case II} will be divided into the following two subcases: 
			\begin{itemize}
				\item \textbf{Case II-1: $l\geq t+1$};
				\item \textbf{Case II-2: $l=t$}.
			\end{itemize} 	
			
			\textbf{Case II-1: $l \geq t + 1$.} It is obvious that 
			\begin{align}\label{equ: est lamnda2}
				|\mathrm{supp}(P(x))| &\geq \binom{n + 2l}{2l + 1} - \binom{n + 2t - 1}{2t} \geq \binom{n + 2t + 2}{n - 1} - \binom{n - 2t - 1}{n - 1} \\
				&\geq \binom{n + 1}{3} + \binom{n + 2}{4} + \binom{n + 3}{5} \geq \binom{n + 2}{3} - 1.\notag
			\end{align}
			
			\textbf{Case II-2: $l = t$.} 
			Let $\{e_j:1\leq j\leq n\}$ be the standard basis of $\R^n$, i.e., $e_j=(0,\cdots,0,1,0,\cdots,0)$ where only the $j$-th place is $1$, and other places are zero. 
			Take $I = 2te_i = (0, \cdots, 2t, 0, \cdots, 0)$. Since $P(x)$ has no negative coefficients, there exists an index $f_i = f_i^j e_j$ such that $f_i + \alpha = \beta + 2t e_i$. Then the components of the multi-index satisfy the equations
			\begin{align*}
				\begin{cases}
					f_i^i + \alpha^i - \beta^i = 2t, \quad 1 \leq i \leq n, \\
					f_i^j + \alpha^j - \beta^j = 0, \quad j \neq i.
				\end{cases}
			\end{align*}
			
			Therefore, $f_i^i - f_j^i = 2t$ for $i \neq j$, which is equivalent to saying that there exists an index $1 \leq k \leq n$ such that $f_i = 2te_i + e_k$, which is equivalent to
			$$
			\beta = \alpha + e_k, \quad 1 \leq k \leq n.
			$$
			When $J + e_k = I$, from $\bm{v} \geq 0$, we get
			\begin{align}\label{equ: lambda}
				\frac{(2l + 1)!}{I!} - \lambda \frac{(2t)!}{J!} \geq 0, |J|=2t.
			\end{align}
			
			The above equation is equivalent to
			$$\lambda \leq (2t + 1)\frac{J!}{I!} = \frac{2t + 1}{j_k + 1},|J|=2t,$$
			where $j_k$  ranges from $0$ to $2t$.
			Thus, $\lambda \leq 1$, and (\ref{equ: lambda}) holds with equality only when $J = 2te_k$. This means
			\begin{align}\label{equ: est lambda 3}
				R(v) = |\mathrm{supp}(P(x))| \geq \binom{n + 2t}{2t + 1} - 1 \geq \binom{n + 2}{3} - 1.
			\end{align}
			Combining (\ref{equ: est lambda1}), (\ref{equ: est lamnda2}) and (\ref{equ: est lambda 3}), we complete the proof of \cref{lem: est gamma rank 2}.
		\end{proof}

		\subsection{The final estimate}

		Proving \cref{thm: main 1}(A) is equal to prove part (a) of \cref{thm:main-1-restate}, now we classify the case  where $J_{n,d}\bm{h} \geq 0$ and $N(\bm{h}) \geq 2$ into six subcases based on the numerical characteristics of the slack variables $\bm{\gamma}$ and $\bm{h}_d$:
		\begin{align*}
			&\textbf{I}: \bm{h}_d = 0. 
			&& \qquad \textbf{II-1}: \bm{h}_d \ngeq 0, R(\bm{\gamma}) \geq 3 
			&& \qquad \textbf{III-1}: \bm{h}_d \geq 0, P(\bm{h}_d) \geq 3, R(\bm{\gamma}) \geq 3. \\
			& 
			&& \qquad \textbf{II-2}: \bm{h}_d \ngeq 0, R(\bm{\gamma}) \leq 2. 
			&& \qquad \textbf{III-2}: \bm{h}_d \geq 0, \bm{h}_d \neq 0, R(\bm{\gamma}) \leq 2. \\
			& 
			&& 
			&& \qquad \textbf{III-3}: \bm{h}_d \geq 0, \bm{h}_d \neq 0, P(\bm{h}_d) \leq 2.
		\end{align*}
		
		We will compute $R(J_{n,d}\bm{h})$ for each of these six categories. The estimates for the first five categories are relatively straightforward, while \textbf{III-3} requires the introduction of patch vectors.
		
		\textbf{I}: $\bm{h}_d = 0$. In this case, the polynomial corresponding to $\bm{h}$ is divisible by a monomial of degree at least one. Therefore,
		\begin{align*}
			R(J_{n,d}\bm{h}) \geq R_{n,d-1}.
		\end{align*}
		
		\textbf{II-1}: $\bm{h}_d \ngeq 0$, $R(\bm{\gamma}) \geq 3$.
		\begin{align*}
			R(J_{n,d}\bm{h}) = R(\bm{\gamma}) + R(J_{n-1,d}\bm{h}_d) \geq 3 + R_{n-1,d}.
		\end{align*}
		
		\textbf{II-2}: $\bm{h}_d \ngeq 0$, $R(\bm{\gamma}) \leq 2$.
		
		There are two possibilities:
		\begin{itemize}
			\item [1.] Only one component of the vector $(R(\bm{\gamma}_0), \cdots, R(\bm{\gamma}_d))^t$ is non-zero. This case has already been studied in \cref{lem: estimate on jnd for gamma rank 1}, and it was shown that $R(J_{n,d}\bm{h}) \geq \binom{n+1}{3} + 1$.
			\item [2.] There exist integers $0 \leq a < b \leq d$ such that $R(\bm{\gamma}_a) = R(\bm{\gamma}_b) = 1$. According to  \cref{lem: est gamma rank 2}, we obtain the same estimate as in the first case:
			\begin{align*}
				R(J_{n,d}\bm{h}) \geq 2 + \binom{n+1}{3} - 1 = \binom{n+1}{3} + 1.
			\end{align*}
		\end{itemize}

		\textbf{III-1}: $\bm{h}_d \geq 0$, $P(\bm{h}_d) \geq 3$, $R(\bm{\gamma}) \geq 3$.
		
		According to   \cref{thm: GHP}, we have $R(J_{n,d}\bm{h}) \geq 3 + 3(n - 1) - 3 \times 2/2 > 3n - 4$.
		
		\textbf{III-2}: $\bm{h}_d \geq 0$, $\bm{h}_d \neq 0$, $R(\bm{\gamma}) \leq 2$.
		
		\cref{lem: estimate on jnd for gamma rank 1} and \cref{lem: est gamma rank 2} do not impose any restrictions on $\bm{h}_d$. Similar to \textbf{II-2}, we have $R(J_{n,d}\bm{h}) \geq \binom{n+1}{3} + 1$.
		
		\textbf{III-3}: $\bm{h}_d \geq 0$, $\bm{h}_d \neq 0$, $P(\bm{h}_d) \leq 2$.
		
		In this case, $1 \leq P(\bm{h}_d) \leq 2$. We introduce a patch vector $\bm{\delta}$ associated with $\bm{h}_d$, which satisfies the following two conditions: (1) $\bm{\delta} \geq 0$, (2) $J_{n-1,d-1}\bm{\delta} \geq \bm{h}_d$. Clearly, such a patch vector always exists for any $\bm{h}_d$. We are interested in the patch with the minimal rank.
		
		If there exists a rank-1 patch $\bm{\delta}$ for $\bm{h}_d$ such that $R(\bm{\delta}) = 1$, by \cref{thm: GHP}, then $R(J_{n-1,d-1}(\bm{\delta}))\leq n-1$. Thus there must exist $\lambda > 0$ such that $\lambda\bm{\delta}$ is also a patch vector for $\bm{h}_d$ and $P(J_{n-1,d-1}(\lambda\bm{\delta}) - \bm{h}_d) \leq n - 2$. Define $\bm{h}' = (\bm{h}_0, \cdots, \bm{h}_{d-2}, \bm{h}_{d-1} + \lambda\bm{\delta})$. By the definition of the patch vector, we have
		$$J_{n-1,d-2}\bm{h}_{d-2} + \bm{h}_{d-1} + \lambda\bm{\delta}, \quad J_{n-1,d-1}(\bm{h}_{d-1} + \lambda\bm{\delta}) \geq 0,$$
		so $J_{n,d-1}\bm{h}' \geq 0$. At this point,
		\begin{align*}
			R(J_{n,d-1}\bm{h}') = & \sum_{i=0}^{d-2}R(\bm{\gamma}_i) + R(J_{n-1,d-2}\bm{h}_{d-2} + \bm{h}_{d-1} + \lambda\bm{\delta}) + R(J_{n-1,d-1}(\bm{h}_{d-1} + \lambda\bm{\delta})) \\
			\leq & \sum_{i=0}^{d-2}R(\bm{\gamma}_i) + P(J_{n-1,d-2}\bm{h}_{d-2} + \bm{h}_{d-1}) + 1 \\
			+ & P(J_{n-1,d-1}\bm{h}_{d-1} + \bm h_d) + P(\lambda J_{n-1,d-1}\bm{\delta} - \bm{h}_d) \\
			\leq & R(\bm{\gamma}) + 1 + n - 2 \leq R(\bm{\gamma}) + R(J_{n-1,d}\bm{h}_d) = R(J_{n,d}\bm{h}),
		\end{align*}
		where the third inequality follows from \cref{thm: Gao-Ng}.
		Since $\sum_{i=0}^{d-1}N(\bm{h}_i) \geq 2$, we have $N(\bm{h}') = \sum_{i=0}^{d-2}N(\bm{h}_i) + N(\bm{h}_{d-1} + \lambda\bm{\delta}) \geq 1$. Therefore, when a rank-1 patch exists,
		$$R(J_{n,d}\bm{h}) \geq R(J_{n,d-1}\bm{h}') \geq R_{n,d-1}.$$
		
		We claim that a rank-1 patch exists when $P(\bm{h}_d) = 1$ or $P(\bm{h}_d) = 2$ and $R(J_{n-1,d}\bm{h}_d) = 2n - 3$. This can be demonstrated within the framework of polynomial multiplication. 
		
		When $P(\bm{h}_d) = 1$, it is obvious. 
		
		When $P(\bm{h}_d) = 2$, to construct the patch vector $\bm{\delta}$, note that there exists a homogeneous polynomial of degree $d$ in $(n-1)$-variables, $A(x_2, \cdots, x_n) = \mathfrak{X}_d\bm{h}_d = a x^\alpha + b x^\beta$, whose coordinates under the left lexicographic order are exactly $\bm{h}_d$, where $a$ and $b$ are positive numbers. Since $R(J_{n-1,d}\bm{h}_d) = 2n - 3$, there exist $2 \leq i, j \leq n$ such that $x^i x^\alpha = x^j x^\beta$. We can take $\bm{\delta}$ to be the coordinate vector corresponding to the $(d-1)$-degree polynomial $(a + b)x^\alpha/x^j$. This completes the proof of the claim.
		
		According to Theorem \cref{thm: GHP}, when $P(\bm{h}_d) = 2$, we have $2n - 3 \leq R(J_{n-1,d}\bm{h}_d) \leq 2n - 2$. Therefore, $R(J_{n-1,d}\bm{h}_d) = 2n - 2$ is the only remaining possibility for \textbf{III-3}. In this case, $\bm{h}_d$ can be written as the sum of two non-negative rank-1 vectors. We use two rank-1 patch vectors $\bm{\delta}_1$ and $\bm{\delta}_2$ to patch these two non-negative rank-1 vectors, respectively. Let $\bm{\delta} = \bm{\delta}_1 + \bm{\delta}_2$. By fine-tuning $\bm{\delta}_1$ and $\bm{\delta}_2$, we ensure that $P(J_{n-1,d-1}\bm{\delta} - \bm{h}_d) \leq 2n - 4$. Still defining $\bm{h}' = (\bm{h}_0, \cdots, \bm{h}_{d-2}, \bm{h}_{d-1} + \bm{\delta})$, similar to the argument for the rank-1 patch, we have
		$$
		R(J_{n,d-1}\bm{h}') \leq R(\bm{\gamma}) + 2 + 2n - 4 = R(\bm{\gamma}) + R(J_{n-1,d}\bm{h}_d) = R(J_{n,d}\bm{h}).
		$$
		If $N(\bm{h}') = \sum_{i=0}^{d-2}N(\bm{h}_i) + N(\bm{h}_{d-1} + \bm{\delta}) \geq 1$, then by definition, we have $R(J_{n,d}\bm{h}) \geq R_{n,d-1}$.
		
		Otherwise, $\sum_{i=0}^{d-2}N(\bm{h}_i) + N(\bm{h}_{d-1} + \bm{\delta}) = 0$, which implies $\sum_{i=0}^{d-2}N(\bm{h}_i) + N(\bm{h}_d) = 0$ and $N(\bm{h}_{d-1}) = 2$. Repeating the argument from Proposition \ref{prop4-1}, we get
		\begin{align*}
			R(J_{n,d}\bm{h}) =& R(\bm{h}_0) + \sum_{i=1}^d R(J_{n-1,i-1}\bm{h}_{i-1} + \bm{h}_i) + R(J_{n-1,d}\bm{h}_d) \\
			\geq & R(\bm{h}_{d-2}) + R(J_{n-1,d-2}\bm{h}_{d-2} + \bm{h}_{d-1}) + R(J_{n-1,d}\bm{h}_d) \\
			\geq & 1 + (n - 1 - 2) + 2(n - 1) = 3n - 4.
		\end{align*}
		Note that  $n \geq 2$, $\binom{n+1}{3} + 1 \geq 3n - 4$. Combining all six subcases, we obtain
		\begin{align}\label{equ: six cases}
			R_{n,d} \geq \min\{R_{n,d-1}, 3 + R_{n-1,d}, 3n - 4\}.
		\end{align}
		In the following, we will complete the proof of \cref{thm: main 1} (A) by induction. The induction idea is illustrated as shown in the following picture.
		$$
		
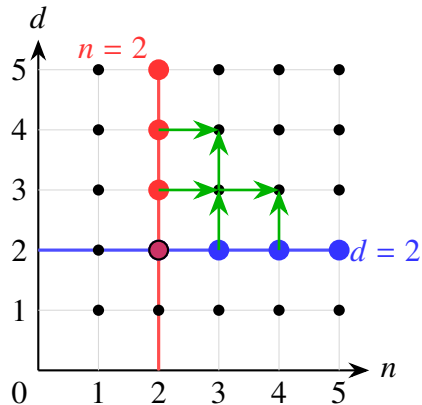
\begin{tikzpicture}[
			scale=0.8,
			point/.style={circle,fill,inner sep=1.5pt},
			boldpoint/.style={circle,fill,inner sep=2.5pt,draw=black,line width=0.8pt},
			axisarrow/.style={-{Stealth[scale=1.2]},thick},
			gridline/.style={gray!30,thin}
			]
			
			\draw[axisarrow] (0,0) -- (5.5,0) node[right] {$n$};
			\draw[axisarrow] (0,0) -- (0,5.5) node[above] {$d$};
			
			\foreach \x in {1,2,3,4,5}
			\node[anchor=north] at (\x,0) {\x};
			\foreach \y in {1,2,3,4,5}
			\node[anchor=east] at (0,\y) {\y};
			
			\foreach \x in {1,2,3,4,5}
			\draw[gridline] (\x,0) -- (\x,5);
			\foreach \y in {1,2,3,4,5}
			\draw[gridline] (0,\y) -- (5,\y);
			
			\draw[line width=1.2pt,red!70] (2,0) -- (2,5); 
			\draw[line width=1.2pt,blue!70] (0,2) -- (5,2); 
			
			\foreach \x in {1,2,3,4,5}
			\foreach \y in {1,2,3,4,5}
			\node[point] at (\x,\y) {};
			
			\foreach \y in {2,3,4,5}
			\node[boldpoint,red!80] at (2,\y) {};
			\foreach \x in {2,3,4,5}
			\node[boldpoint,blue!80] at (\x,2) {};
			
			\node[boldpoint,fill=purple!80] at (2,2) {};
			
			\draw[-{Stealth[scale=1.3]},line width=1pt,green!70!black] (3,2) -- (3,3); 
			\draw[-{Stealth[scale=1.3]},line width=1pt,green!70!black] (2,3) -- (3,3); 
			\draw[-{Stealth[scale=1.3]},line width=1pt,green!70!black] (3,3) -- (4,3); 
			\draw[-{Stealth[scale=1.3]},line width=1pt,green!70!black] (3,3) -- (3,4); 
			\draw[-{Stealth[scale=1.3]},line width=1pt,green!70!black] (2,4) -- (3,4); 
			\draw[-{Stealth[scale=1.3]},line width=1pt,green!70!black] (4,2) -- (4,3); 
			
			\node[below left] at (0,0) {$0$};
			\node[above left,red!80] at (2,5) {$n=2$};
			\node[right,blue!80] at (5,2) {$d=2$};
		\end{tikzpicture}
		$$
		
		In the figure, we want to show the bold lattice points $(n, d)$ satisfy $R_{n,d} \geq 3n - 4$, for $n,d\geq 2$. 
		
		For   $n = 2$, as previously discussed, \cref{thm: Gao-Ng} gives $R_{2,d} \geq 2 = 3 \times 2  - 4$. For $d = 2$, we have proved   that  
		$R_{3,2} \geq 5$, $R_{4,2} \geq 8$, $R_{5,2} \geq 11$, which satisfies $R_{n,2}\geq 3n-4$, and \cref{thm: main 1} (B) shows $R_{n,2}\geq 3n-4$ for $n\geq 6$.
		
		Assuming $R_{n-1,d} \geq 3n - 7$ and $R_{n,d-1} \geq 3n - 4$, for $n,d\geq 3$. As illustrated by the arrows in the picture, we only need to prove $R_{n,d}\geq 3n-4$.
		
		From (\ref{equ: six cases}), we have 
		\begin{align*}
			R_{n,d} &\geq \min\{R_{n,d-1}, 3 + R_{n-1,d}, 3n - 4\}\\
			&\geq \min\{3n - 4, 3 + 3n - 7, 3n - 4\} = 3n - 4.
		\end{align*}
		\textbf{This completes the proof of \cref{thm: main 1} (A).}
		
		
		\section{Proof of \cref{Eben's conj low dim}}
		
		In this section,  we proceed to prove the \cref{Eben's conj low dim}. Since any  diagonal Hermitian polynomial \( A(z, \bar{z}) \) of degree at most \( 2d \),  as well as its first prolongation, can be written as a sum of diagonal Hermitian bihomogeneous polynomials:
		\[
		A(z, \bar{z})\|z\|^2 = \sum_{0 \leq i \leq d} A_i(z, \bar{z})\|z\|^2, \ A_i \in P_{n,i}.
		\]  
		Thus \( A_i(z, \bar{z})\|z\|^2 \in \mbox{SOS}_n \)  if and only if \( A_i(z, \bar{z})\|z\|^2 \in \Sigma_{n,i+1} \) for \( 0 \leq i \leq d \), and  
		\[
		R(A(z, \bar{z})\|z\|^2) = \sum_{0 \leq i \leq d} R(A_i(z, \bar{z})\|z\|^2) = \sum_{0 \leq i \leq d} R(J_{n,i}\bm{a}_i),
		\]  
		where \( \bm{a}_i \) denotes the coordinate vector of the diagonal polynomial \( A_i (z, \bar{z})\) under the basis \( \mathfrak{X}_i \) as previously described.  
		
		By the second inequality of \cref{lem: counting ineq}, if the coordinate vector corresponding to \( A_1(z, \bar{z}) \) has negative components, then \( A_1(z, \bar{z})\|z\|^2 \notin \Sigma_{n,2} \), and the same holds for \( A_0(z, \bar{z})\|z\|^2 \). Thus when \( A(z, \bar{z}) \notin \mbox{SOS}_n\) and \(A(z, \bar{z})||z||^2 \in \mbox{SOS}_n \) , there must exist an integer \( 2 \leq k \leq d \) such that \( A_k(z, \bar{z}) \notin \Sigma_{n,k} , A_k(z, \bar{z})\|z\|^2 \in \Sigma_{n,k+1} \). By \cref{thm: main 1}(A) 
		\[
		R(A(z, \bar{z})\|z\|^2) \geq R(A_k(z, \bar{z})\|z\|^2) \geq R_{n,k} \geq 3n - 4.
		\]  
		
		Let \( k_i := R(\bm{a}_i) \). When \( A(z, \bar{z}) \in \mbox{SOS}_n \) , by \cref{thm: GHP} either
		\[
		n\sum_{i=0}^d k_i - \frac{\left(\sum_{i=0}^d k_i\right)\left(\sum_{i=0}^d k_i - 1\right)}{2} \leq n\sum_{i=0}^d k_i - \sum_{i=0}^d \frac{k_i(k_i - 1)}{2} \leq R(A(z, \bar{z})\|z\|^2) \leq n\sum_{i=0}^d k_i,
		\]  
		or
		\[R(A(z, \bar{z})\|z\|^2) \geq \frac{n(n+1)}{2}. \]
		By direct computation, when \( 2 \leq n \leq 6 \), regardless of whether \( A(z, \bar{z}) \) is positive semi-definite, if the polynomial \( A(z, \bar{z}) \) becomes positive semi-definite after the first extension, the rank \( R \) of \( A(z, \bar{z})\|z\|^2 \) either satisfies  
		\[
		R \geq (\kappa_0 + 1)n - \frac{(\kappa_0 + 1)\kappa_0}{2} - 1,
		\]  
		where \( \kappa_0 \) is the largest integer such that \( \kappa(\kappa + 1)/2 < n \), or there exists \( \kappa \in \{0, 1, 2, \cdots, \kappa_0\} \) such that  
		\[
		\kappa n - \frac{\kappa(\kappa - 1)}{2} \leq R \leq \kappa n.
		\]

	\end{document}